\documentclass[11pt,oneside,leqno]{amsart}
\usepackage{amsxtra}
\usepackage{amsopn}
\usepackage{color}
\usepackage{amsmath,amsthm,amssymb}
\usepackage{amscd}
\usepackage{amsfonts}
\usepackage{latexsym}
\usepackage{verbatim}
\usepackage{pb-diagram}
\usepackage{color}

\usepackage{booktabs}
\usepackage{multirow}
\usepackage{array}
\usepackage[all]{xy}

\theoremstyle{plain}
\newtheorem{theorem}{Theorem}[section]
\newtheorem*{theorem*}{Theorem}
\newtheorem{definition}[theorem]{Definition}

\newtheorem{prop}[theorem]{Proposition}
\newtheorem*{prop*}{Proposition}
\newtheorem{cor}[theorem]{Corollary}
\newtheorem*{cor*}{Corollary}
\newtheorem{rem}[theorem]{Remark}
\newtheorem{ex}[theorem]{Example}
\newtheorem*{mt*}{Main Theorem}
\DeclareMathOperator{\Imm}{Im}

\newcommand{\rr}{\mathbb{R}}
\newcommand{\cc}{\mathbb{C}}
\newcommand{\del}{\partial}
\newcommand{\delbar}{\overline{\del}}

\begin{document}
\title[On the cohomology of almost-complex and symplectic manifolds...]{On the cohomology of almost-complex and symplectic manifolds and proper surjective maps}
\author{Nicoletta Tardini and Adriano Tomassini}
\date{\today}
\address{Dipartimento di Matematica\\
Universit\`{a} di Pisa \\
Largo Bruno Pontecorvo 5, 56127 \\
Pisa, Italy}
\email{tardini@mail.dm.unipi.it}
\address{Dipartimento di Matematica e Informatica\\
Universit\`{a} di Parma \\
Parco Area delle Scienze 53/A, 43124 \\
Parma, Italy}
\email{adriano.tomassini@unipr.it}
\thanks{Partially supported by GNSAGA
of INdAM}
\begin{abstract}
Let $(X,J)$ be an almost-complex manifold. In \cite{li-zhang} Li and Zhang introduce
$H^{(p,q),(q,p)}_J(X)_{\rr}$ as the cohomology subgroups of the $(p+q)$-th de Rham cohomology group formed by classes represented by real pure-type forms.
Given a proper, surjective, pseudo-holomorphic map between two almost-complex manifolds we study the relationship among such cohomology groups.
Similar results are proven in the symplectic setting for the cohomology groups
introduced in \cite{tsengyauI} by Tseng and Yau
and a new characterization of the Hard Lefschetz condition
 in dimension $4$ is provided.
\end{abstract}

\maketitle
\section*{Introduction}

On an almost-complex manifold $(X,J)$, Li and Zhang in \cite{li-zhang} introduce
the following subgroups of the de Rham cohomology
$$
H^{(p,q),(q,p)}_J(X)_{\rr}:=\left\lbrace [\alpha]\in H^{p+q}_{dR}(X;\rr)\vert\alpha\in \left(A^{p,q}_J(X)\oplus A^{q,p}_J(X)\right)\cap
A^{p+q}(X;\rr)\right\rbrace
$$
where, in particular, $H^{(1,1)}_J(X)_{\rr}$ and $H^{(2,0),(0,2)}_J(X)_{\rr}$ represent the $J$-invariant and $J$-anti-invariant cohomology groups. They have been introduced in order to study the relations between the tamed cone $\mathcal{K}_J^t$ and the compatible cone $\mathcal{K}_J^c$ on compact almost-complex manifolds.\\
More precisely, if $(X,J)$ is a compact almost-complex manifold admitting a compatible symplectic structure, i.e., $\mathcal{K}_J^c\neq\emptyset$, and $J$ is
$\mathcal{C}^\infty$-full (that is $H^2_{dR}(X;\rr)=H^{(1,1)}_J(X)_{\rr}+H^{(2,0),(0,2)}_J(X)_{\rr}$) then Li and Zhang prove in \cite[Theorem 1.1]{li-zhang} that
$$
\mathcal{K}_J^t=\mathcal{K}_J^c+H^{(2,0),(0,2)}_J(X)_{\rr}.
$$
Such a decomposition generalizes the following relation for compact K\"ahler manifolds
(see \cite[Corollary 3.2]{li-zhang})
$$
\mathcal{K}_J^t=\mathcal{K}_J^c+ \left(\left(H^{(2,0)}_{\overline\partial}(X)\oplus H^{(0,2)}_{\overline\partial}(X)\right)\cap H^2_{dR}(X;\rr)\right).
$$
In this sense $H^{(2,0),(0,2)}_J(X)_{\rr}$ can be viewed as a generalization of
$H^{(2,0)}_{\overline\partial}(X)\oplus H^{(0,2)}_{\overline\partial}(X)$ to
the non integrable case.

Recently in the symplectic setting Tseng and Yau in \cite{tsengyauI} introduce
new cohomology groups which can be viewed as the symplectic counterpart
of the Bott-Chern and the Aeppli cohomology groups on a complex manifold, namely for a symplectic manifold $(X,\omega)$, denoting by $d^\Lambda$ the
symplectic codifferential, they define
\[
H^k_{d+d^\Lambda}\left(X\right)
:=\frac{\ker(d+d^\Lambda)\cap A^k(X)}{\Imm dd^\Lambda\cap A^k(X)},
\]
\[
H^k_{dd^\Lambda}\left(X\right)
:=\frac{\ker(dd^\Lambda)\cap A^k(X)}{\left(\Imm d+\Imm d^\Lambda\right)\cap A^k(X)}
\]
as the symplectic Bott-Chern and Aeppli cohomology groups of $(X,\omega)$ respectively.\\
The aim of this paper is to study the relations among these groups under proper, surjective pseudo-holomorphic/symplectic-structure-preserving maps.
In the differentiable category in \cite{wells} Wells shows that if we have a proper surjective differentiable map between two manifolds, under more suitable hypothesis, we can compare the de Rham cohomology groups.
In particular, (\cite[Theorem 3.3]{wells})
let $\pi:\tilde X\longrightarrow X$ be a surjective proper differentiable map between two orientable, differentiable manifolds of the same dimension. If $\hbox{deg}\,\pi\neq 0$
then the induced map on cohomology
$$
\pi^*:H^\bullet_{dR}(X;\rr)\longrightarrow H^\bullet_{dR}(\tilde X;\rr)
$$
is injective.\\
Hence if $\tilde X$ and $X$ are compact, then we have the inequalities
$b_\bullet(X)\leq b_\bullet(\tilde X)$ on the Betti numbers of $X$ and $\tilde X$, respectively.\\
Moreover, for complex manifolds a natural holomorphic invariant is furnished by the Dolbeault cohomology groups and a similar result still holds, as proved by Wells (see \cite[Theorem 3.1]{wells}).
Namely, if we ask $\tilde X$ and $X$ to be complex manifolds of the same complex dimension and $\pi:\tilde X\longrightarrow X$ to be a surjective proper holomorphic map (we do not need to ask $\hbox{deg}\,\pi\neq 0$ anymore) then we have injections on the complex-valued de Rham cohomology groups and the Dolbeault cohomology groups, i.e.,
$$
\pi^*:H^\bullet_{dR}(X;\cc)\longrightarrow H^\bullet_{dR}(\tilde X;\cc)
$$
and
$$
\pi^*:H^{\bullet,\bullet}_{\overline\partial}(X)\longrightarrow
H^{\bullet,\bullet}_{\overline\partial}(\tilde X).
$$
are injective maps.\\
We recall that on a complex manifold $X$ we can consider other complex cohomology groups, which, in general, are not isomorphic to the Dolbeault cohomology, namely
the \emph{Bott-Chern cohomology groups} and the \emph{Aeppli cohomology groups} defined respectively as (see \cite{aeppli})
\[
H_{BC}^{\bullet,\bullet}(X):=
\frac{\hbox{\rm Ker}\del\cap
\hbox{\rm Ker}\overline{\partial}}{\hbox{\rm Im}\del\delbar},\qquad
H_{A}^{\bullet,\bullet}(X):=
\frac{\hbox{\rm Ker}\del\delbar}
{\hbox{\rm Im}\del+
\hbox{\rm Im}\delbar}.
\]
They can not be seen as the cohomology groups associated to a resolution of a sheaf, nevertheless they can be computed using currents as in the case of the de Rham and Dolbeault cohomology (cf. \cite{schweitzer}).
For these cohomologies in \cite[Theorem 3.1]{angella}, Angella proves that $\pi:\tilde X\longrightarrow X$ induces injections
$$
\pi^*:H^{\bullet,\bullet}_{BC}(X)\longrightarrow
H^{\bullet,\bullet}_{BC}(\tilde X),\qquad
\pi^*:H^{\bullet,\bullet}_{A}(X)\longrightarrow
H^{\bullet,\bullet}_{A}(\tilde X).
$$
Therefore, if $\tilde X$ and $X$ are compact then we have the inequalities
$h^{p,q}_{\delbar}(X)\leq h^{p,q}_{\delbar}(\tilde X)$,
$h^{p,q}_{BC}(X)\leq h^{p,q}_{BC}(\tilde X)$ and
$h^{p,q}_{A}(X)\leq h^{p,q}_{A}(\tilde X)$ where
$h^{p,q}_{\sharp}(X):=\dim_{\cc} H^{p,q}_{\sharp}(X)$,
$\sharp\in\left\lbrace\delbar, BC, A\right\rbrace$, for any $p,q$.\\
If $\tilde X$ and $X$ have different dimension the injectivity of $\pi^*$ among the Dolbeault cohomology groups is false in general (as shown by Wells considering the projection of the Hopf surface into $\mathbb{P}^1(\cc)$), unless $\tilde X$ is K\"ahler (or symplectic if we consider only the de Rham cohomology). Indeed, (see
\cite[ Theorem 4.1, Theorem 4.3]{wells})
if $\pi:\tilde X^{2m}\longrightarrow X^{2n}$ is a surjective, proper, differentiable map between two even-dimensional, orientable, differentiable manifolds and if $\tilde X$ admits a symplectic structure $\tilde \omega$ and $\pi_*(\tilde\omega^{m-n})\neq 0$ then the induced map on cohomology
$$
\pi^*:H^\bullet_{dR}(X;\rr)\longrightarrow H^\bullet_{dR}(\tilde X;\rr).
$$
is injective.\\
Moreover, if $\tilde X$ is a K\"ahler manifold, $X$ is a complex manifold and $\pi$ is
a holomorphic map then
$$
\pi^*:H^\bullet_{dR}(X;\cc)\longrightarrow H^\bullet_{dR}(\tilde X;\cc),\qquad
\pi^*:H^{\bullet,\bullet}_{\overline\partial}(X)\longrightarrow
H^{\bullet,\bullet}_{\overline\partial}(\tilde X).
$$
are injective maps.
The function $\mu=\pi_*(\tilde\omega^{m-n})$ is called by Wells in \cite{wells}
the \emph{symplectic degree of $\pi$} and it depends on the choice of the symplectic form on $\tilde X$.\\
We focus our attention on the cohomology groups $H^{(p,q),(q,p)}_J(X)_{\rr}$
and $H^k_{d+d^\Lambda}\left(X\right)$.
We prove the following results (see section \ref{results}):
\begin{theorem}
Let $\pi:(\tilde X,\tilde J)\longrightarrow (X,J)$ be a proper, surjective, pseudo-holomorphic map between two almost-complex manifolds of the same dimension.
Then,
$$
\pi^*:H^{(p,q),(q,p)}_J(X)_{\rr}\longrightarrow H^{(p,q),(q,p)}_{\tilde J}(\tilde X)_{\rr}
$$
is injective for any $p,q$.\\
In particular, if $\tilde X$ and $X$ are compact we have the inequalities
$$h^{(p,q),(q,p)}_J(X)\leq h^{(p,q),(q,p)}_{\tilde J}(\tilde X)$$
for any $p,q$.
\end{theorem}
This Theorem can also be seen as a generalization of Proposition 4.3 in \cite{zhang} which involves only the $J$-invariant and $J$-anti-invariant cohomology groups.\\
For almost-complex manifolds of different dimension we prove the following
\begin{theorem}
Let $\pi:(\tilde X^{2m},\tilde J)\longrightarrow (X^{2n},J)$ be a proper, surjective, pseudo-holomorphic map between two almost-complex manifolds and suppose that
$(\tilde\omega,\tilde J)$ is an almost-K\"ahler structure on $\tilde X$.
Then,
$$
\pi^*:H^{(p,q),(q,p)}_J(X)_{\rr}\longrightarrow H^{(p,q),(q,p)}_{\tilde J}(\tilde X)_{\rr}.
$$
is injective for any $p,q$.\\
In particular, if $\tilde X$ and $X$ are compact we have the inequalities
$$h^{(p,q),(q,p)}_J(X)\leq h^{(p,q),(q,p)}_{\tilde J}(\tilde X)$$
for any $p,q$.
\end{theorem}
In section \ref{example} we provide an example
showing that the almost-K\"ahler assumption in the theorem above can not 
be dropped.\\
We obtain a similar result in the symplectic case under the further assumption
that one of the involved manifolds satisfies the Hard Lefschetz condition (see Theorem \ref{symplectic}). Furthermore we give a characterization of this condition for a compact symplectic $4$-manifold in terms of symplectic Bott-Chern numbers
(see Theorem \ref{numeri}), namely
\begin{theorem}
Let $(X^4,\omega)$ be a compact symplectic $4$-manifold, then
it satisfies the Hard Lefschetz condition if and only if
\[
b_2(X)=\dim H^2_{d+d^\Lambda}(X).
\]
\end{theorem}
This result is the symplectic analogue of a result proved in \cite{angella-dlousski-tomassini} for compact complex surfaces. We also do explicit computations
of the symplectic cohomologies on compact $4$-dimensional solvmanifolds.

\noindent {\em Acknowledgements.} The authors would like to thank
Tian-Jun Li and Weiyi Zhang for useful comments and for pointing out the reference \cite{zhang}.

\smallskip

\section{Preliminaries}\label{preliminaries}

We start by fixing some notations and recalling some well-known results about relations between cohomologies of manifolds related by proper surjective maps.\\
If $Y$ is an orientable differential manifold of dimension $m$,
we denote by $A^r(Y)$ the space of differential $r$-forms on $Y$ and by $\mathcal{D}_{m-r}(Y)$ the space of currents of dimension $m-r$ (or of degree $r$)
on $Y$, i.e. the topological dual of the space of differential $(m-r)$-forms with compact support in $Y$.\\ 
Let $\tilde X$ and $X$ be two orientable, differentiable manifolds of the same dimension $m$
and let $\pi:\tilde X\longrightarrow X$ be a surjective proper differentiable map between them. We recall that the \emph{degree of $\pi$} is defined as $deg\pi:=\pi_*(1)$ where $\pi_*$ is the map induced by $\pi$ on top-dimension currents
$$
\pi_*:\mathcal{D}_m(\tilde X)\longrightarrow \mathcal{D}_m(X).
$$
If $\pi$ is orientation preserving then $deg\pi\neq 0$, in particular it is positive
(cf. \cite{wells}).
If we consider the diagram
$$
\xymatrixcolsep{5pc}\xymatrix{
A^{r}(\tilde X) \ar[r]^{\tilde i} &
\mathcal{D}_{m-r}(\tilde X)\ar[d]^{\pi_*}\\
A^{r}(X) \ar[u]^{\pi^*}\ar[r]^{i}   &
\mathcal{D}_{m-r}(X)
}
$$
Wells prove in \cite[Lemma 2.3]{wells} that this is commutative up to the degree of $\pi$, namely $\mu i=\pi_*\tilde i\pi^*$ with $\mu:=deg\pi$. Since in the induced diagram on (co)homology
$$\xymatrixcolsep{5pc}\xymatrix{
H^{r}_{dR}(\tilde X;\rr) \ar[r]^{\tilde i_*} &
H_{m-r}(\tilde X;\rr)\ar[d]^{\pi_*}\\
H^{r}_{dR}(X;\rr) \ar[u]^{\pi^*}\ar[r]^{i_*}   &
H_{m-r}(X;\rr);
}
$$
the maps $\tilde i_*$ and $i_*$ are isomorphisms Wells proves that
$$
\pi^*:H^\bullet_{dR}(X;\rr)\longrightarrow H^\bullet_{dR}(\tilde X;\rr)
$$
is injective.\\
Similar arguments in the complex case show a comparison between the Dolbeault, Bott-Chern and Aeppli cohomology groups.\\
In order to give similar results in the almost-complex case we consider two differentiable manifolds $\tilde X$ and $X$ of the same dimension $2n$, endowed with two (non necessarily integrable) almost complex structures $\tilde J$ and
$J$, respectively. We fix the orientation induced by $\tilde J$ on $\tilde X$ and by $J$ on $X$. Let $\pi:(\tilde X,\tilde J)\longrightarrow (X,J)$ be a proper, surjective, pseudo-holomorphic map between the two considered manifolds.
Notice that, by the psudo-holomorphic assumption, $\mu:=deg\pi\neq 0$, indeed $\pi$ preserves the orientation. Therefore, under this hypothesis, the map $\pi$ induces an injection on the de-Rham cohomology groups (\cite[Theorem 3.3]{wells})
$$
\pi^*:H^\bullet_{dR}(X;\rr)\longrightarrow H^\bullet_{dR}(\tilde X;\rr).
$$  
In particular, if $\tilde X$ and $X$ are compact, we have the inequalities $b_\bullet(X)\leq b_\bullet(\tilde X)$ on the Betti numbers.\\
If we consider the bi-grading induced by the almost-complex structures on the complexes of forms $\left(A^{\bullet,\bullet}_{\tilde J}(\tilde X),d\right)$ and $\left(A^{\bullet,\bullet}_{J}(X),d\right)$, we have that $\pi$ is bigrading preserving, i.e.,
$$
\pi^*:A^{p,q}_J(X)\longrightarrow A^{p,q}_{\tilde J}(\tilde X).
$$
and
$$
\pi_*:\mathcal{D}_{n-p,n-q}^{\tilde J}(\tilde X)\longrightarrow \mathcal{D}_{n-p,n-q}^{J}(X).
$$
for any $p,q$.\\
We set (see \cite{fino-tomassini})
$$
A^{(p,q),(q,p)}_{J}(X)_{\rr}:=\left(A^{p,q}_J(X)\oplus A^{q,p}_J(X)\right)\cap
A^{p+q}(X;\rr)
$$
and
$$
H^{(p,q),(q,p)}_J(X)_{\rr}:=\left\lbrace [\alpha]\in H^{p+q}_{dR}(X;\rr)\vert\alpha\in A^{(p,q),(q,p)}_{J}(X)_{\rr}\right\rbrace;
$$
a similar definition (and notation) can be given for currents. If $X$ is compact, set
$h^{(p,q),(q,p)}_J(X):=\dim H^{(p,q),(q,p)}_J(X)_{\rr}$.\\
In particular, $H^{(1,1)}_J(X)_{\rr}$ and $H^{(2,0),(0,2)}_J(X)_{\rr}$ represent the $J$-invariant and $J$-anti-invariant part of the second de Rham cohomology group of $X$ (see \cite{li-zhang} for further details), where $J$ acts on the space of $k$-forms in the following way:
$$
J\alpha(v_1,\ldots,v_k):=\alpha(Jv_1,\ldots,Jv_k).
$$
By commutation relations with the differential operator $d$ on forms and currents the maps $\pi^*$ and $\pi_*$ can be induced in (co)homology obtaining
$$
\pi^*:H^{(p,q),(q,p)}_J(X)_{\rr}\longrightarrow H^{(p,q),(q,p)}_{\tilde J}(\tilde X)_{\rr}
$$
and
$$
\pi_*:H_{(n-p,n-q),(n-q,n-p)}^{\tilde J}(\tilde X)_{\rr}\longrightarrow H_{(n-p,n-q),(n-q,n-p)}^{J}(X)_{\rr}.
$$
We recall the following definition
\begin{definition}
An almost-complex structure $J$ on a differential manifold  $X$ is called
\begin{itemize}
\item[-] \emph{$\mathcal{C}^\infty$-pure} if 
\[
H_J^{(1,1)}(X)_{\rr}\cap H_J^{(2,0)(0,2)}(X)_{\rr}=\left\lbrace 0\right\rbrace,
\]
\item[-]  \emph{$\mathcal{C}^\infty$-full} if
\[
H^2_{dR}(X;\rr)= H_J^{(1,1)}(X)_{\rr}+H_J^{(2,0)(2,0)}(X)_{\rr},
\]
\item[-] \emph{$\mathcal{C}^\infty$-pure and full} if it is $\mathcal{C}^\infty$-pure and
$\mathcal{C}^\infty$-full, i.e., if
\[
H^2_{dR}(X;\rr)=H_J^{(1,1)}(X)_{\rr}\oplus H_J^{(2,0)(2,0)}(X)_{\rr}.
\]
\end{itemize}
\end{definition}
Examples of $\mathcal{C}^\infty$-pure and full structures are K\"ahler structures.
More in general, Draghici, Li and Zhang prove that any
almost complex structure on a compact $4$-manifold is $\mathcal{C}^\infty$-pure and full (see \cite[Theorem 2.3]{draghici-li-zhang}). 

\section{The almost-complex case}\label{results}

We start by proving the natural generalization of Wells' Theorem in this context.
\begin{theorem}\label{injectivity}
Let $\pi:(\tilde X,\tilde J)\longrightarrow (X,J)$ be a proper, surjective, pseudo-holomorphic map between two almost-complex manifolds of the same dimension.
Then,
$$
\pi^*:H^{(p,q),(q,p)}_J(X)_{\rr}\longrightarrow H^{(p,q),(q,p)}_{\tilde J}(\tilde X)_{\rr}
$$
is injective for any $p,q$.\\
In particular, if $\tilde X$ and $X$ are compact we have the inequalities
$$h^{(p,q),(q,p)}_J(X)\leq h^{(p,q),(q,p)}_{\tilde J}(\tilde X)$$
for any $p,q$.
\end{theorem}
\begin{proof}
Consider the following diagram
$$\xymatrixcolsep{5pc}\xymatrix{
A^{(p,q),(q,p)}_{\tilde J}(\tilde X)_{\rr} \ar[r]^{\tilde i} &
\mathcal{D}_{(n-p,n-q),(n-q,n-p)}^{\tilde J}(\tilde X)_{\rr}\ar[d]^{\pi_*}\\
A^{(p,q),(q,p)}_{J}(X)_{\rr} \ar[u]^{\pi^*}\ar[r]^{i}   &
\mathcal{D}_{(n-p,n-q),(n-q,n-p)}^{J}(X)_{\rr};
}
$$
this is commutative up to the degree of $\pi$, indeed we are just restricting the diagram considered in Section \ref{preliminaries}
$$
\xymatrixcolsep{5pc}\xymatrix{
A^{r}(\tilde X) \ar[r]^{\tilde i} &
\mathcal{D}_{m-r}(\tilde X)\ar[d]^{\pi_*}\\
A^{r}(X) \ar[u]^{\pi^*}\ar[r]^{i}   &
\mathcal{D}_{m-r}(X)
}
$$
to elements in $A^{(p,q),(q,p)}_{J}(X)_{\rr}$. Therefore, we have
that the equality $\mu i=\pi_*\tilde i\pi^*$ (where $\mu=deg\pi$) holds, and this can be induced in (co-)homology in the following diagram
$$\xymatrixcolsep{5pc}\xymatrix{
H^{(p,q),(q,p)}_{\tilde J}(\tilde X)_{\rr} \ar[r]^{\tilde i_*} &
H_{(n-p,n-q),(n-q,n-p)}^{\tilde J}(\tilde X)_{\rr}\ar[d]^{\pi_*}\\
H^{(p,q),(q,p)}_{J}(X)_{\rr} \ar[u]^{\pi^*}\ar[r]^{i_*}   &
H_{(n-p,n-q),(n-q,n-p)}^{J}(X)_{\rr}.
}
$$
The maps $\tilde i_*$ and $i_*$ in this last diagram are injective, indeed they are induced by the quasi-isomorphism
$A^\bullet(Y)\longrightarrow \mathcal{D}_{2n-\bullet}(Y)$ given by $\varphi\mapsto \int_Y\varphi\wedge\cdot$ with $Y=X,\tilde X$.\\
Using this fact we have that $\pi^*$ is injective, indeed
let $a\in H^{(p,q),(q,p)}_{J}(X)_{\rr}$ and
suppose that $\pi^*a=0$, then $\mu i_*a=\pi_*\tilde i_*\pi^*a=0$. Since $\mu\neq 0$, then $i_*a=0$ and by injectivity we can conclude that $a=0$, proving the assertion.
\end{proof}
As a consequence for $(p,q)=(1,1)$ and $(p,q),(q,p)=(2,0),(0,2)$, namely the $J$-invariant and the $J$-anti-invariant cases, we recover the result proved by Zhang in  \cite[Proposition 4.3]{zhang}. More precisely, we get the following
\begin{cor}
\cite[Proposition 4.3]{zhang}
Let $\pi:(\tilde X,\tilde J)\longrightarrow (X,J)$ be a surjective, pseudo-holomorphic map between two compact almost-complex manifolds of the same dimension.
Then, $h^+_J(X)\leq h^+_{\tilde J}(\tilde X)$ and $h^-_J(X)\leq h^-_{\tilde J}(\tilde X)$.
\end{cor}
We can now prove the following proposition which generalizes Proposition 3.3 in \cite{fino-tomassini} given in the integrable context.
\begin{prop}
Let $\pi:(\tilde X,\tilde J)\longrightarrow (X,J)$ be a proper, surjective, pseudo-holomorphic map between two almost-complex manifolds of the same dimension.
If $\tilde J$ is $\mathcal{C}^\infty$-pure than $J$ is $\mathcal{C}^\infty$-pure too.
\end{prop}
\begin{proof}
As already observed
$$
\pi^*:H^\bullet_{dR}(X;\rr)\longrightarrow H^\bullet_{dR}(\tilde X;\rr)
$$
is injective and bi-grading preserving.
By contradiction, assume that there exists $a\in H^{(1,1)}_J(X)_{\rr}\cap H^{(2,0),(0,2)}(X)_{\rr}$, $a\neq 0$. Hence $a=[\alpha]=[\beta]$ with $\alpha\in A_J^{(1,1)}(X)_{\rr}\cap \ker d$
and $\beta\in A_J^{(2,0),(2,0)}(X)_{\rr}\cap \ker d$,
and $\pi^*\alpha=\pi^*\beta+d(\pi^*\gamma)$ on $\tilde X$. By injectivity, $0\neq [\pi^*\alpha]=[\pi^*\beta]\in
H^{(1,1)}_J(X)_{\rr}\cap H^{(2,0),(0,2)}(X)_{\rr}$ and this is absurd, since $\tilde J$ is $\mathcal{C}^\infty$-pure by hypothesis.
\end{proof}
\begin{rem}
A similar result can be obtained considering \emph{pure} almost complex structures, in the sense of currents and \emph{complex-$\mathcal{C}^{\infty}$-pure} almost complex structures (cf. \cite{li-zhang} for the precise definitions).
\end{rem}
If $\tilde X$ and $X$ have different dimension the injectivity of $\pi^*$ is no more true in general, as we will show in Example \ref{etabeta}.
Nevertheless if $\tilde X$ has a symplectic form compatible with the almost complex structure the injectivity of $\pi^*$ is guaranteed. Indeed, if we consider the case when the manifolds $\tilde X$ and $X$ have different dimensions, say $2m$ and $2n$ respectively, we can use a compatible symplectic form on $(\tilde X, \tilde J)$ to carry currents of type $(m-p,m-q)$ on $\tilde X$ to currents of type $(n-p,n-q)$ on $X$. We prove the following
\begin{theorem}\label{diversadimensione}
Let $\pi:(\tilde X^{2m},\tilde J)\longrightarrow (X^{2n},J)$ be a proper, surjective, pseudo-holomorphic map between two almost-complex manifolds and suppose that
$(\tilde\omega,\tilde J)$ is an almost-K\"ahler structure on $\tilde X$.
Then,
$$
\pi^*:H^{(p,q),(q,p)}_J(X)_{\rr}\longrightarrow H^{(p,q),(q,p)}_{\tilde J}(\tilde X)_{\rr}.
$$
is injective for any $p,q$.\\
In particular, if $\tilde X$ and $X$ are compact we have the inequalities
$$h^{(p,q),(q,p)}_J(X)\leq h^{(p,q),(q,p)}_{\tilde J}(\tilde X)$$
for any $p,q$.
\end{theorem}
\begin{proof}
Set $d:=m-n>0$ and notice that the map induced by $\pi$ on currents is given by
$$
\pi_*:\mathcal{D}_{m-p-d,m-q-d}^{\tilde J}(\tilde X)_{\rr}\longrightarrow \mathcal{D}_{n-p,n-q}^{J}(X)_{\rr}.
$$
As in \cite{wells} we define the map
$$
\tau:\mathcal{D}_{m-p,m-q}^{\tilde J}(\tilde X)_{\rr}\longrightarrow \mathcal{D}_{n-p,n-q}^{J}(X)_{\rr}
$$
as $\tau(T):=\pi_*(T\wedge\tilde\omega^d)$; this is well defined since
$\tilde\omega$ is of type $(1,1)$ with respect to the bigrading induced by $\tilde J$
on forms.\\
We recall that, in general, if $\tilde X$ admits a symplectic form, then there exists a
constant $\mu$ (the symplectic degree of $\pi$) such that
$\mu i=\tau\tilde i\pi^*$ in the diagram (\cite{wells})
$$
\xymatrixcolsep{5pc}\xymatrix{
A^{r}(\tilde X) \ar[r]^{\tilde i} &
\mathcal{D}_{m-r}(\tilde X)\ar[d]^{\tau}\\
A^{r}(X) \ar[u]^{\pi^*}\ar[r]^{i}   &
\mathcal{D}_{n-r}(X).
}
$$
In particular, if we restrict to real $(p,q)+(q,p)$-forms we get that the same equality holds in the diagram
$$\xymatrixcolsep{5pc}\xymatrix{
A^{(p,q),(q,p)}_{\tilde J}(\tilde X)_{\rr} \ar[r]^{\tilde i} &
\mathcal{D}_{(m-p,m-q),(m-q,m-p)}^{\tilde J}(\tilde X)_{\rr}\ar[d]^{\tau}\\
A^{(p,q),(q,p)}_{J}(X)_{\rr} \ar[u]^{\pi^*}\ar[r]^{i}   &
\mathcal{D}_{(n-p,n-q),(n-q,n-p)}^{J}(X)_{\rr}.
}
$$
Since $\pi^*$ and $\tau$ commute with the differential operator $d$, we have that the previous equality still holds passing to (co)homology
$$\xymatrixcolsep{5pc}\xymatrix{
H^{(p,q),(q,p)}_{\tilde J}(\tilde X)_{\rr} \ar[r]^{\tilde i_*} &
H_{(m-p,m-q),(m-q,m-p)}^{\tilde J}(\tilde X)_{\rr}\ar[d]^{\tau}\\
H^{(p,q),(q,p)}_{J}(X)_{\rr} \ar[u]^{\pi^*}\ar[r]^{i_*}   &
H_{(n-p,n-q),(n-q,n-p)}^{J}(X)_{\rr}.
}
$$
With the same argument used when $\tilde X$ and $X$ have the same dimension we can conclude that $\pi^*$ is injective.
\end{proof}
\begin{rem}
With similar considerations we can obtain the same results considering bi-graded complex-valued forms and, in this case, when the almost complex structure $J$ on $X$
is integrable $H^{(p,q)}_{J}(X)_{\cc}$ is the the image of
the natural map $H^{p,q}_{BC}(X)\longrightarrow H^{p+q}_{dR}(X;\cc)$ induced by the identity.
\end{rem}

\section{Example}\label{example}

The following example shows that the almost-K\"ahler assumption in Theorem \ref{diversadimensione} is crucial.
\begin{ex}\label{etabeta}
Consider the holomorphically parallelizable  complex nilmanifold of real dimension $10$ defined in \cite{alessandrini-bassanelli} as the quotient of
$$
G:=\left\lbrace \left[\begin{matrix}
1 & z^1 & z^3 & z^5\\
0 & 1 & 0 & z^2\\
0 & 0 &1 & z^4\\
0 & 0 & 0 & 1
\end{matrix}\right]
\mid z^1,z^2,z^3, z^4, z^5\in\cc\right\rbrace
$$
and the subgroup $\Gamma$ of matrices with entries in $\mathbb{Z}[i]$. The complex structure equations on $\eta\beta_5:=\frac{G}{\Gamma}$ are
\[
\left\{\begin{array}{rcl}
            d\varphi^1 &=&   0 \\[5pt]
            d\varphi^2 &=&   0  \\[5pt]
            d\varphi^3 &=&   0  \\[5pt]
            d\varphi^4 &=&   0  \\[5pt]
            d\varphi^5 &=&   -\varphi^1\wedge\varphi^2-\varphi^3\wedge\varphi^4
       \end{array}\right. \;.
\]
We recall that $\eta\beta_5$ is a $\mathcal{C}^\infty$-pure-and-full manifold that does not admit any K\"ahler structure and we want to show that the conclusion in Theorem \ref{diversadimensione} does not hold in this case.\\
Indeed, we can consider the natural projection on the first four coordinates
$$
\pi:\eta\beta_5\longrightarrow \mathbb{T}^4_{\cc};
$$ this is a proper surjective holomorphic map if we consider the torus with the standard complex structure $J_0$. We recall that in \cite{angella-tomassini} the $J$-invariant and $J$ anti-invariant numbers of $\eta\beta_5$ are computed: in particular $h_J^{(1,1)}(\eta\beta_5)=16$ and
$h_J^{(2,0)(0,2)}(\eta\beta_5)=10$. It is easy to see that
$h_{J_0}^{(1,1)}(\mathbb{T}^4_{\cc})=16$ and
$h_{J_0}^{(2,0)(0,2)}(\mathbb{T}^4_{\cc})=12$. Hence we have $h_J^{(2,0)(0,2)}(\eta\beta_5)<
h_{J_0}^{(2,0)(0,2)}(\mathbb{T}^4_{\cc})$ showing that
$$
\pi^*:H^{(2,0),(0,2)}_{J_0}(\mathbb{T}^4_{\cc})_{\rr}\longrightarrow H^{(2,0),(0,2)}_{\tilde J}(\eta\beta_5)_{\rr}
$$
is not injective.
\end{ex}

\section{Symplectic cohomologies}

On a symplectic manifold it is possible to define other cohomology groups which are strictly related to the symplectic structure. Therefore we ask if, under suitable hypothesis, a proper surjective map between two symplectic manifolds induces an injection on these symplectic cohomology groups. We give a partial answer to this question requiring that at least one of the two manifolds satisfies the Hard Lefschetz condition.
In order to better understand the following results we start by recalling some known facts from complex geometry. Let $X$ be a compact complex manifold,
then in \cite{angella-tomassini2} Angella and the second author define the following integers
$$
\Delta^k:=\sum_{p+q=k}h^{p,q}_{BC}(X)+\sum_{p+q=k}h^{p,q}_{A}(X)-2b_k(X)\geq0,   \qquad k\in\mathbb{Z},
$$
showing that their triviality characterizes the $\del\delbar$-lemma. Moreover, when $X$ is a compact complex surface in \cite{angella-dlousski-tomassini} and \cite{angella-tomassini-verbitsky} it is proven that $\Delta^1=0$ and $\Delta^2\in\left\lbrace 0,2\right\rbrace$. Since on compact complex surfaces K\"ahlerianity is equivalent to the $\del\delbar$-lemma holding on $X$, an immediate consequence is that a compact complex surface is K\"ahler if and only if $\Delta^2=0$.
Furthermore for a compact complex surface $X$ K\"ahlerianity is topologically characterized by the first Betti number $b_1$ (namely $b_1$ is even if and only if $X$ is K\"ahler if and only if $\Delta^2=0$, or equivalently $b_1$ is odd if and only if $X$ is non-K\"ahler if and only if $\Delta^2=2$). Since $b_1$ is a bimeromorphic invariant for complex manifolds we get the following
\begin{theorem}\label{modification}
Let $\pi:\tilde X\longrightarrow X$ a proper modification between two compact complex surfaces, then $\Delta^k$ is invariant under the map $\pi$.
\end{theorem}
This Theorem is not true in higher dimension; the first example is due to Hironaka  who constructed a non-K\"ahler modification of $\cc\mathbb{P}^3$.\\
Let $(X,\omega)$ be a symplectic manifold, then Tseng and Yau in \cite{tsengyauI} define a symplectic version of the Bott-Chern and the Aeppli cohomology groups. If we denote with
$\star:A^\bullet(X)\longrightarrow A^{2n-\bullet }(X)$
the \emph{symplectic-$\star$-Hodge operator} (see \cite{brylinski}) and with
$\Lambda:A^\bullet(X)\longrightarrow A^{\bullet -2}(X)$ the adjoint of the
Lefschetz operator $L:A^\bullet(X)\longrightarrow A^{\bullet +2}(X)$,
the \emph{Brylinski co-differential} is defined as
\[
d^\Lambda:=\left[d,\Lambda\right]=d\Lambda-\Lambda d=(-1)^{k+1}\star d\star.
\]
Then, the \emph{$d^\Lambda$-cohomology groups} are
\[
H^k_{d^\Lambda}\left(X\right)
:=\frac{\ker(d^\Lambda)\cap A^k(X)}{\Imm d^\Lambda\cap A^k(X)},
\]
the \emph{symplectic Bott-Chern cohomology groups} are
\[
H^k_{d+d^\Lambda}\left(X\right)
:=\frac{\ker(d+d^\Lambda)\cap A^k(X)}{\Imm dd^\Lambda\cap A^k(X)}
\]
and the \emph{symplectic Aeppli cohomology groups} are
\[
H^k_{dd^\Lambda}\left(X\right)
:=\frac{\ker(dd^\Lambda)\cap A^k(X)}{\left(\Imm d+\Imm d^\Lambda\right)\cap A^k(X)}.
\]
If we consider a compatible triple $(\omega,J,g)$ on $X$
(meaning that the almost-complex structure $J$ is $\omega$-calibrated and $g$ is the respective Riemannian metric on $X$) then, denoting with $*$ the standard \emph{Hodge-operator}
with respect to the Riemannian metric $g$,
there are canonical isomorphisms (see \cite{tsengyauI})
\[
\mathcal{H}^k_{d^\Lambda}\left(X\right):=\ker\Delta_{d^\Lambda}
\simeq H^k_{d^\Lambda}\left(X\right),
\]
where $\Delta_{d^{\Lambda}} :=  d^{\Lambda*}d^\Lambda+d^\Lambda d^{\Lambda*}$ is a second-order elliptic self-adjoint
differential operator and
\[
\mathcal{H}^k_{d+d^\Lambda}\left(X\right):=\ker\Delta_{d+d^\Lambda}
\simeq H^k_{d+d^\Lambda}\left(X\right),\qquad
\mathcal{H}^k_{dd^\Lambda}\left(X\right):=\ker\Delta_{dd^\Lambda}
\simeq H^k_{dd^\Lambda}\left(X\right).
\]
where
$\Delta_{d+d^{\Lambda}}$, $\Delta_{dd^{\Lambda}}$ are fourth-order elliptic self-adjoint
differential operators defined by
\[
\begin{array}{lcl}
\Delta_{d+d^{\Lambda}}& := &(dd^{\Lambda})(dd{^\Lambda})^*+(dd^{\Lambda})^*(dd^{\Lambda})+
d^*d^{\Lambda} d^{\Lambda *}d+d^{\Lambda *}d d^*d^{\Lambda}+d^*d+d^{\Lambda *}d^{\Lambda},\\[10pt]
\Delta_{dd^{\Lambda}} & :=& (dd^{\Lambda})(dd{^\Lambda})^*+(dd^{\Lambda})^*(dd^{\Lambda})+
dd^{\Lambda *}d^\Lambda d^*+d^\Lambda d^*dd^{\Lambda *}+dd^*
+d^\Lambda d^{\Lambda *}.
\end{array}
\]
In particular, the symplectic cohomology groups are finite-dimensional vector spaces on a compact
symplectic manifold. For
$\sharp\in\left\{d^\Lambda,d+d^\Lambda,
dd^\Lambda\right\}$ we denote $h^\bullet_\sharp:=:h^\bullet_\sharp(X):= \dim H^\bullet_\sharp(X)<\infty$ when the manifold $X$ is understood.\\
Historically the differential forms closed both for the operators
$d$ and $d^\Lambda$ were called
\emph{symplectic harmonic} (\cite{brylinski}).
The existence of a symplectic harmonic form in each de Rham
cohomology class requires that the \emph{hard-Lefschetz property} holds on $(X^{2n},\omega)$
(cf. \cite{mathieu}), i.e., 
\[
L^k:H^{n-k}_{dR}(X)\longrightarrow H^{n+k}_{dR}(X), \qquad 0\leq k\leq n
\]
are all isomorphisms, and also the uniqueness of a symplectic harmonic representative
does not occur in general.
In particular, the following facts are equivalent on a compact symplectic
manifold $(X^{2n},\omega)$ (cf. \cite{brylinski}, \cite{mathieu}, \cite{merkulov},
\cite{yan}, \cite{cavalcanti})
\begin{itemize}
\item{} the hard-Lefschetz condition (HLC for short) holds;
\item{} the \emph{Brylinski conjecture}, i.e., the existence of a
symplectic harmonic form in each de Rham cohomology class;
\item{} the \emph{$dd^\Lambda$-lemma}, i.e., every $d^\Lambda$-closed, $d$-exact form
is also $dd^\Lambda$-exact;
\item{} the natural maps induced by the identity $H^{\bullet}_{d+d^\Lambda}(X)
\longrightarrow H^{\bullet}_{dR}(X)$ are injective;
\item{} the natural maps induced by the identity $H^{\bullet}_{d+d^\Lambda}(X)
\longrightarrow H^{\bullet}_{dR}(X)$ are surjective;
\item{} the natural maps induced by the identity in the following diagram are isomorphisms
$$ \xymatrix{
  & H^{\bullet}_{d+d^\Lambda}(X) \ar[ld]\ar[rd] & \\
  H^{\bullet}_{dR}(X) \ar[rd] &  & H^{\bullet}_{d^\Lambda}(X). \ar[ld] \\
  & {\phantom{\;.}} H^{\bullet}_{dd^\Lambda}(X) \; &
} $$
\end{itemize}
In this sense $H^{\bullet}_{d+d^\Lambda}\left(X\right)$ and
$H^{\bullet}_{dd^\Lambda}\left(X\right)$ represent more appropriate cohomologies
talking about existence and uniqueness of harmonic representatives.\\
Neverthless, in general, on a symplectic manifold of dimension $2n$ the following maps
are all isomorphisms (see \cite{tsengyauI} Prop. 3.24)
$$\xymatrixcolsep{5pc}\xymatrix{
\mathcal{H}^{k}_{d+d^\Lambda}\left(X\right) \ar[d]^{L^{n-k}} \ar[r]^{*} &
\mathcal{H}^{2n-k}_{dd^\Lambda}\left(X\right)\ar[d]^{\Lambda^{n-k}}\\
\mathcal{H}^{2n-k}_{d+d^\Lambda}\left(X\right) \ar[r]^{*}   &
\mathcal{H}^{k}_{dd^\Lambda}\left(X\right),
}
$$
in particular, it follows that $h^k_{d+d^\Lambda}=h^{2n-k}_{d+d^\Lambda}=
h^k_{dd^\Lambda}=h^{2n-k}_{dd^\Lambda}$ for all $k=0,\ldots, 2n$.

\subsection{Characterization of the $dd^\Lambda$-lemma}

In \cite{angella-tomassini3} Angella and the second author, starting from a purely
algebraic point of view,
introduce on a compact symplectic manifold $(X^{2n},\omega)$
the following integers 
\[
\Delta^k:= h^k_{d+d^\Lambda}+h^k_{dd^\Lambda}-2b_k\geq 0,
\qquad k\in\mathbb{Z},
\]
proving that, similarly to the complex case, their triviality characterizes the $dd^\Lambda$-lemma.
In this sense these numbers measure the HLC-degree of a symplectic manifold,
as their analogue in the complex case do (cf. \cite{angella-tomassini2}).\\
Now, as already observed in \cite{chansuen},
using the equality $\dim H^\bullet_{d+d^\Lambda}=
\dim H^\bullet_{dd^\Lambda}$ proved in \cite{tsengyauI}, we can write the non-HLC degrees
as follows
\[
\Delta^k=2(h^k_{d+d^\Lambda}-b_k),\qquad k\in\mathbb{Z};
\]
therefore we can simplify them, considering just the difference between the dimensions of the
Bott-Chern and the de Rham cohomology groups. We define
\[
\tilde\Delta^k:=h^k_{d+d^\Lambda}-b_k,\qquad k\in\mathbb{Z}.
\]
Notice that a similar simplification can not be done in the complex case (cf. \cite{schweitzer}).
We put in evidence that, by duality, $\tilde\Delta^k=\tilde\Delta^{2n-k}$,
$k=0,\ldots 2n$; for a compact symplectic manifold $(X,\omega)$
of dimension $2n$ we will refer
to $\tilde\Delta^k$, $k=0\ldots n$, as the \emph{non-HLC-degrees} of $X$.
Note that $\tilde\Delta^0=0$.\\
As a consequence of the positivity of $\Delta^k$, for any $k$, we have that
for all $k=1,\ldots,n$
\[
b_k\leq h^k_{d+d^\Lambda}
\]
on a compact symplectic $2n$-dimensional manifold.\\
Moreover the equalities 
\[
b_k=h^k_{d+d^\Lambda}, \qquad \forall k=1,\ldots,n,
\]
hold on a compact symplectic $2n$-dimensional manifold if and only if it 
satisfies the Hard-Lefschetz condition; namely 
the equality $b_\bullet= h^\bullet_{d+d^\Lambda}$ ensures
the bijectivity of the natural maps $H^\bullet_{d+d^\Lambda}(X)\longrightarrow
H^\bullet_{dR}(X)$, and hence the $dd^\Lambda$-lemma.\\
This Theorem can be inserted in the more general setting of generalized complex manifolds,
see \cite{chansuen} for more details.
\begin{rem}
Note that, as proved in \cite{angellatardini}, on a compact complex manifold
the equality between the dimensions of the Bott-Chern cohomology groups
and the Aeppli cohomology groups characterizes the $\del\delbar$-lemma;
nevertheless, the "analogous" condition on a compact symplectic manifold $X$, namely
$h^{\bullet}_{d+d^\Lambda}(X)=
h^{\bullet}_{dd^\Lambda}(X)$, is
always verified.
\end{rem}
Similarly to the complex case where $\Delta^2$ characterizes the K\"ahlerianity
of a compact complex surface, if $2n=4$ we want
to show that the only degree which characterizes HLC is $\tilde\Delta^2$. Indeed we
show the following
\begin{theorem}
Let $(X^{2n},\omega)$ be a compact symplectic manifold, then
$$\tilde\Delta^1=0.$$
\end{theorem}
\begin{proof}
First of all we prove that the natural map induced
by the identity
\[
H^1_{d+d^\Lambda}(X)\longrightarrow H^1_{dR}(X)
\]
is an isomorphism.\\
The surjectivity follows from \cite[Lemma 2.7]{debartolomeis-tomassini}.
For the sake of completeness we briefly recall here the proof. Let $\alpha$ be
a $d$-closed $1$-form, then
$$
d^\Lambda\alpha=\left[d,\Lambda\right]\alpha=-\Lambda d\alpha=0,
$$
i.e., $\alpha$ is also $d^\Lambda$-closed, proving the surjectivity.
We need to prove the injectivity.\\
Let $a=[\alpha]\in H^1_{d+d^\Lambda}(X)$ be such that $a=0$ in $H^1_{dR}(X)$, i.e.,
$\alpha=df$ for some $f\in\mathcal{C}^\infty(X)$.
Considering the Hodge decomposition of $f$ with respect to the $d^\Lambda$-cohomology (cf. \cite{tsengyauI})
we get $f=c+d^\Lambda\beta$ with $c$ constant and $\beta$ differential
$1$-form. Hence
\[
\alpha=df=d(c+d^\Lambda\beta)=dd^\Lambda\beta,
\]
i.e., $[\alpha]=0$ in $H^1_{d+d^\Lambda}(X)$.\\
As a consequence, $b_1=h^1_{d+d^\Lambda}$, implying
$\tilde\Delta^1=h^1_{d+d^\Lambda}-b_1=0$.
\end{proof}
\begin{rem}
We put in evidence the fact that $\tilde\Delta^1=0$ on a compact symplectic manifold of any dimension.
In the complex setting the analogue result holds only in complex dimension $2$
(see \cite{angella-dlousski-tomassini}).
\end{rem}
As a consequence we get the following
\begin{theorem}\label{numeri}
Let $(X^4,\omega)$ be a compact symplectic $4$-manifold, then
it satisfies
\[
HLC \iff \tilde\Delta^2=0 \iff b_2=h^2_{d+d^\Lambda}.
\]
\end{theorem}

\subsubsection{Example}

As shown in \cite{angella-tomassini-verbitsky} on a compact complex surface $\Delta^2\in\left\lbrace 0,2\right\rbrace$; we now provide an example of a compact symplectic $4$-manifold with
$\Delta^2\notin\left\lbrace 0,2\right\rbrace$, or equivalently $\tilde\Delta^2\notin\left\lbrace 0,1\right\rbrace$, showing hence a different behavior in the symplectic case.
More precisely, we compute the non-HLC degree $\tilde\Delta^2$ when $X$ is a compact
$4$-dimensional manifold diffeomorphic to
a solvmanifold $\Gamma\backslash G$ (i.e., the compact quotient of a connected
simply-connected solvable Lie group $G$ by a discrete cocompact subgroup
$\Gamma$) admitting a
left-invariant symplectic structure; for a partial computation cfr.
\cite[Table 2]{angellakasuya}.\\
According to \cite{bock} we have the following cases
\begin{itemize}
\item[a)] the \emph{primary Kodaira surface}
$\mathfrak{g}=\mathfrak{g}_{3,1}\oplus\mathfrak{g}_1=(0,0,0,23)$;
\item[b)] $\mathfrak{g}=\mathfrak{g}_{1}\oplus\mathfrak{g}_{3,4}^{-1}=
(0,0,-23,24)$;
\item[c)] $\mathfrak{g}=\mathfrak{g}_{4,1}=(0,0,12,13)$;
\item[d)] the \emph{torus} $\mathfrak{g}=4\mathfrak{g}_{1}=(0,0,0,0)$;
\item[e)] the \emph{hyper-ellptic surface}
$\mathfrak{g}=\mathfrak{g}_{1}\oplus\mathfrak{g}_{3,5}^0=(0,0,-24,23)$.
\end{itemize}
where $\mathfrak{g}$ is the Lie Algebra of $G$. The last two solvmanifolds admit
a K\"ahler structure.
\begin{theorem}
Let $X=\Gamma\backslash G$ be a compact solvmanifold of dimension $4$ with $\omega$
left-invariant symplectic structure.
Then, according to $\mathfrak{g}=Lie(G)$,
\begin{itemize}
\item[a)] if $\mathfrak{g}=\mathfrak{g}_{3,1}\oplus\mathfrak{g}_1$, then
$\tilde\Delta^2=1$;
\item[b)] if $\mathfrak{g}=\mathfrak{g}_{1}\oplus\mathfrak{g}_{3,4}^{-1}$, then
$\tilde\Delta^2=0$;
\item[c)] if $\mathfrak{g}=\mathfrak{g}_{4,1}$, then
$\tilde\Delta^2=2$.
\end{itemize}
In particular, every left-invariant symplectic structure in case $b)$ satisfies $HLC$.\\
Moreover, none of the symplectic structures in case $a)$ and $c)$ satisfy $HLC$ as
we expected, indeed the Lie algebra in these two cases is nilpotent
(\cite{benson-gordon}).
\end{theorem}
\begin{proof}
In each case we will denote with $\left\{e^1,\ldots, e^6\right\}$ a left-invariant co-frame
on $\Gamma\backslash G$. For brevity of notations we put $e^{12}:=e^1\wedge e^2$
and so on.\\
\textbf{a)} 
The structure equations on the primary Kodaira surface are
\[
\left\{\begin{array}{rcl}
            d e^1 &=&   0 \\[5pt]
            d e^2 &=&   0  \\[5pt]
            d e^3 &=&   0  \\[5pt]
            d e^4 &=&   e^{23}  \\[5pt]
       \end{array}\right. \;.
\]
Consider an arbitrary left-invariant $d$-closed $2$-form
\[
\omega=\lambda_{12}e^{12}+\lambda_{13}e^{13}+\lambda_{23}e^{23}+
\lambda_{24}e^{24}+\lambda_{34}e^{34};
\]
$\omega$ is symplectic if and only if $\lambda_{12}\lambda_{34}+
\lambda_{13}\lambda_{24}\neq 0$.
According to \cite{angellakasuya} we can compute the symplectic cohomology groups by just using
invariant forms, in particular we obtain $h^2_{d+d^\Lambda}=5$. Therefore
$$
\tilde\Delta^2=h^2_{d+d^\Lambda}-b_2=5-4=1.
$$
\textbf{b)}
The structure equations in this case are
\[
\left\{\begin{array}{rcl}
            d e^1 &=&   0 \\[5pt]
            d e^2 &=&   0  \\[5pt]
            d e^3 &=&   -e^{23}  \\[5pt]
            d e^4 &=&   e^{24}  \\[5pt]
       \end{array}\right. \;.
\]
Consider an arbitrary left-invariant $d$-closed $2$-form
\[
\omega=\lambda_{12}e^{12}+\lambda_{23}e^{23}+
\lambda_{24}e^{24}+\lambda_{34}e^{34};
\]
$\omega$ is symplectic if and only if $\lambda_{12}\neq 0$ and $\lambda_{34}\neq 0$.
With the previous argument we get $h^2_{d+d^\Lambda}=2$. Therefore
$$
\tilde\Delta^2=h^2_{d+d^\Lambda}-b_2=2-2=0.
$$
\textbf{c)}
The structure equations in this case are
\[
\left\{\begin{array}{rcl}
            d e^1 &=&   0 \\[5pt]
            d e^2 &=&   0  \\[5pt]
            d e^3 &=&   e^{12}  \\[5pt]
            d e^4 &=&   e^{13}  \\[5pt]
       \end{array}\right. \;.
\]
Consider an arbitrary left-invariant $d$-closed $2$-form
\[
\omega=\lambda_{12}e^{12}+\lambda_{13}e^{13}+\lambda_{14}e^{14}+
\lambda_{23}e^{23};
\]
$\omega$ is symplectic if and only if $\lambda_{14}\neq 0$ and $\lambda_{23}\neq 0$.
With the previous argument we get $h^2_{d+d^\Lambda}=4$. Therefore
$$
\tilde\Delta^2=h^2_{d+d^\Lambda}-b_2=4-2=2.
$$
\end{proof}

\subsection{Behavior under proper surjective maps}

As already recalled, K\"ahlerianity (or equivalently the $\del\delbar$-lemma) is invariant under modification of compact complex surfaces; we give some conditions under which that the $dd^\Lambda$-lemma is invariant under proper, surjective, symplectic-structure-preserving maps between $4$-dimensional symplectic manifolds.
First of all we prove the following
\begin{prop}\label{HLC}
Let $\pi:(\tilde X^{2n},\tilde\omega)\longrightarrow (X^{2n},\omega)$ be a smooth, proper, surjective map such that $\pi^*\omega=\tilde\omega$. If
$\tilde X$ satisfies HLC then $X$ satisfies HLC.
\end{prop}
\begin{proof}
Consider the following diagram
$$\xymatrixcolsep{5pc}\xymatrix{
H^{n-k}_{dR}(X;\rr) \ar[r]^{\left[\omega^k\right]} \ar[d]^{\pi^*}&
H^{n+k}_{dR}(X;\rr)\ar[d]^{\pi^*}\\
H^{n-k}_{dR}(\tilde X;\rr) \ar[r]^{\left[\tilde\omega^k\right]}   &
H^{n+k}_{dR}(\tilde X;\rr);
}
$$
this is commutative since, by hypothesis, $\pi^*(\omega^k)=\tilde\omega^k$.
By Wells' theorem the maps $\pi^*$ on the left and on the right of the diagram are injective for any $k$. Then the bijectivity of $\left[\tilde\omega^k\right]$ forces the bijectivity of
$\left[\omega^k\right]$.
\end{proof}
We ask whether there exists an example of
a smooth, proper, surjective map
$\pi:(\tilde X^{2n},\tilde\omega)\longrightarrow (X^{2n},\omega)$ with $\tilde X$ which does not satisfy HLC but $X$ does.\\
We prove that the converse of Proposition \ref{HLC} is true when $\tilde X$ and $X$ are $4$-dimensional manifolds with the same first Betti number. This hypothesis in Theorem \ref{modification} is unnecessary since $b_1$ is invariant under bimeromorphisms
and K\"ahlerianity is in fact characterized in terms of $b_1$.
\begin{prop}
Let $\pi:(\tilde X^{4},\tilde\omega)\longrightarrow (X^{4},\omega)$ be a smooth, proper, surjective map such that $\pi^*\omega=\tilde\omega$. 
Suppose that $b_1(\tilde X)=b_1(X)$, then
$\tilde X$ satisfies HLC if and only if $X$ satisfies HLC.
\end{prop}
\begin{proof}
Consider the following commutative diagram
$$\xymatrixcolsep{5pc}\xymatrix{
H^{1}_{dR}(X;\rr) \ar[r]^{\left[\omega\right]} \ar[d]^{\pi^*}&
H^{3}_{dR}(X;\rr)\ar[d]^{\pi^*}\\
H^{1}_{dR}(\tilde X;\rr) \ar[r]^{\left[\tilde\omega\right]}   &
H^{3}_{dR}(\tilde X;\rr)
}
$$
By Wells' theorem the maps $\pi^*$ on the left and on the right of the diagram are injective. Moreover, by the hypothesis on $b_1$, we have that the map $\pi^*$ on the left is an isomorphism.\\
Suppose that $X$ satisfies HLC, we have to prove that $\tilde X$ satisfies HLC too.
Let $\left[\alpha\right]\in H^{1}_{dR}(\tilde X;\rr)$ such that
$0=\left[\tilde\omega\wedge\alpha\right]\in H^{3}_{dR}(\tilde X;\rr)$. Then,
$0=\left[\tilde\omega\wedge\alpha\right]=
\pi^*\left(\left[\omega\right]\cup\left(\pi^{*^{-1}}\left[\alpha\right]\right)\right)$, hence
$0=\left[\omega\right]\cup\left(\pi^{*^{-1}}\left[\alpha\right]\right)$
since $\pi^*:H^3_{dR}(X)\longrightarrow H^3_{dR}(\tilde X)$ is injective.
The manifold $X$ satisfies HLC, so $\pi^{*^{-1}}\left[\alpha\right]=0$ and therefore $\left[\alpha\right]=0$, indeed
$\pi^*:H^1_{dR}(X)\longrightarrow H^1_{dR}(\tilde X)$ is bijective.
\end{proof}
We ask whether the assumption on $b_1$ can be dropped.\\
The following is an analogue of Theorem \ref{injectivity} in this context.
\begin{theorem}\label{symplectic}
Let $\pi:(\tilde X^{2n},\tilde\omega)\longrightarrow (X^{2n},\omega)$ be a smooth, proper, surjective map between two equi-dimensional symplectic manifolds such that $\pi^*\omega=\tilde\omega$. If $X$ satisfies the $dd^\Lambda$-lemma then
$$
\pi^*:H^k_{d+d^\Lambda}(X)\longrightarrow H^k_{d+d^\Lambda}(\tilde X)
$$
is injective for any $k$.
\end{theorem}
\begin{proof}
The hypothesis $\pi^*\omega=\tilde\omega$ guarantees that the map $\pi^*$ is well-defined. Consider the following commutative diagram
$$\xymatrixcolsep{5pc}\xymatrix{
H^{k}_{d+d^\Lambda}(X) \ar[r]^{\pi^*_{d+d^\Lambda}} \ar[d]^{i}&
H^{k}_{d+d^\Lambda}(\tilde X)\ar[d]^{\tilde i}\\
H^{k}_{dR}(X;\rr) \ar[r]^{\pi^*_{dR}}   &
H^{k}_{dR}(\tilde X;\rr)
}
$$
where $i$, $\tilde i$ are the natural maps induced by the identity and
we denote with $\pi^*_{d+d^\Lambda}$,
$\pi^*_{dR}$ the maps induced by $\pi$ on the two cohomologies. If $X$ satisfies
the $dd^\Lambda$-lemma then $i$ is bijective.
Fix $k$ and take $\left[\alpha\right]\in H^{k}_{d+d^\Lambda}(X)$ such that
$0=\left[\pi^*_{d+d^\Lambda}(\alpha)\right]\in H^{k}_{d+d^\Lambda}(\tilde X)$.
Then $0=\tilde i\left(\left[\pi^*_{d+d^\Lambda}(\alpha)\right]\right)
=
\pi^*_{dR}\left(i\left(\left[\alpha\right]\right)\right)\in H^{k}_{dR}(\tilde X;\rr)$.
Since by Wells' Theorem $\pi^*_{dR}$ is injective,
$0=i\left(\left[\alpha\right]\right)\in H^{k}_{dR}(X;\rr)$, hence
$0=\left[\alpha\right]\in H^{k}_{d+d^\Lambda}(X)$ showing the injectivity
of $\pi^*_{d+d^\Lambda}$.
\end{proof}
Of course, if $\tilde X$ satisfies the $dd^\Lambda$-lemma the same result hold,
by using Proposition \ref{HLC}.\\
It could be interesting to understand whether it is possible to avoid the HLC assumption.\\


\begin{thebibliography}{12}


\bibitem{aeppli} A. Aeppli, On the cohomology structure of Stein manifolds,
{\em Proc. Conf. Complex Analysis (Minneapolis, Minn., 1964)}
Springer, Berlin, 1965, 58--70.


\bibitem{alessandrini-bassanelli} L. Alessandrini, G. Bassanelli, Small deformations of a class of compact non-K\"ahler manifolds,
{\em Proc. Amer. Math. Soc.} \textbf{109} (1990), 1059--1062.


\bibitem{angella} D. Angella, Cohomologies of certain orbifolds,
{\em J. Geom. Phys.} \textbf{71} (2013), 117--126.


\bibitem{angella-dlousski-tomassini}
D. Angella, G. Dloussky, A. Tomassini,
On Bott-Chern cohomology of compact complex surfaces,
\texttt{arXiv:1402.2408 [math.DG]}.

\bibitem{angellakasuya} D. Angella, H. Kasuya, 
Symplectic Bott-Chern cohomology of solvmanifolds,
 \texttt{arXiv:1308.4258 [math.SG]}.

\bibitem{angellatardini} D. Angella, N. Tardini, 
Quantitative and qualitative cohomological properties for non-K\"ahler
manifolds, \texttt{	arXiv:1507.07108 [math.CV]}.


\bibitem{angella-tomassini3}
D. Angella, A. Tomassini, Inequalities \emph{\`a la} Fr\"ohlicher
and cohomological decompositions, to appear in
{\em J. Noncommut. Geom.}


\bibitem{angella-tomassini} D. Angella, A. Tomassini, On the cohomology of almost-complex manifolds,
{\em Int. J. Math.} \textbf{23} (2012), 1250019, 25 pp.


\bibitem{angella-tomassini2}
D. Angella, A. Tomassini, On the $\del\delbar$-lemma and Bott-Chern cohomology,
{\em Invent. Math.} {\bfseries 192} (2013), no.~1, 71--81.

\bibitem{angella-tomassini-verbitsky} D. Angella, A. Tomassini, M. Verbitsky,
On non-K\"ahler degrees of complex manifolds, in preparation.

\bibitem{benson-gordon} Ch. Benson, C. S. Gordon, K\"ahler and symplectic
structures on nilmanifolds,
{\em Topology} \textbf{27} (1988), no.~4, 513--518.

\bibitem{bock} Ch. Bock, On low-dimensional solvmanifolds,
\texttt{arXiv:0903.2926v4 [math.DG]}.

\bibitem{brylinski} J.-L. Brylinski, A differential complex for Poisson manifolds,
{\em J. Differ. Geom.} \textbf{28} (1988), no.~1, 93--114.


\bibitem{cavalcanti} G. R. Cavalcanti, New aspects of the $dd^c$-lemma,
Oxford University D. Phil thesis, \texttt{arXiv:math/0501406v1 [math.DG]}.


\bibitem{chansuen} K. Chan, Y.-H. Suen, A Fr\"olicher-type inequality for generalized complex manifolds, \emph{Ann. Global Anal. Geom.} \textbf{47} (2015), no.~2, 135--145.

\bibitem{debartolomeis-tomassini} P. de Bartolomeis, A. Tomassini,
On formality of some symplectic manifolds,
{\em Int. Math. Res. Not.} \textbf{24} (2001), no.~24, 1287--1314.


\bibitem{draghici-li-zhang} T. Draghici, T.-J- Li, W. Zhang, Symplectic forms and cohomology decomposition of almost complex four-manifolds, {\em IMRN} \textbf{2010} (2010), 1--17.

\bibitem{fino-tomassini} A. Fino, A. Tomassini, On blow-ups and cohomology of almost complex manifolds,
{\em Diff. Geom. Appl.} \textbf{30} (2012), 520--529.


\bibitem{li-zhang} T.-J. Li, W. Zhang, Comparing tamed and compatible symplectic cones and cohomological
properties of almost complex manifolds, {\em Comm. Anal. Geom.} \textbf{17} (2009), 651--683.


\bibitem{mathieu} O. Mathieu, Harmonic cohomology classes of symplectic manifolds,
{\em Comment. Math. Helv.} \textbf{70} (1995), no.~1, 1--9.

\bibitem{merkulov} S. A. Merkulov, Formality of canonical symplectic complexes and Frobenius
manifolds,
{\em Int. Math. Res. Not.} \textbf{1998} (1998), no.~14, 727--733.


\bibitem{schweitzer} M. Schweitzer, Autour de la cohomologie de Bott-Chern, {\tt preprint} arXiv:0709.3528v1[math. AG].

\bibitem{tsengyauI} L.-S. Tseng, S.-T. Yau, Cohomology and Hodge Theory on Symplectic
manifolds: I,
{\em J. Differ. Geom.} \textbf{91} (2012), no.~3, 383--416.

\bibitem{wells} R. O. Wells, Comparison of de Rham and Dolbeault cohomology for proper surjective mappings,
{\em Pacific J. Math.} \textbf{53} (1974), 281--300.

\bibitem{yan} D. Yan, Hodge structure on symplectic manifolds,
{\em Adv. Math.} \textbf{120} (1996), no.~1, 143--154.

\bibitem{zhang} W. Zhang, Geometric structures, Gromov norm and Kodaira dimensions, {\tt preprint} arXiv:1404.4231 [math.GT].

\end{thebibliography}
\end{document}